\documentclass[10pt]{article}
\usepackage{amsmath}
\usepackage{mathrsfs}
\usepackage{amsfonts}
\usepackage{amsmath,amsthm}
\usepackage{amsmath,amssymb,amsthm,latexsym}
\usepackage{graphics}
\usepackage{subfigure}
\usepackage{float}
\usepackage{graphicx}
\usepackage{amscd}
\usepackage[all]{xy}
\usepackage{hyperref}

\newtheorem{theorem}{Theorem}[section]

\newtheorem{definition}[theorem]{Definition}
\newtheorem{corollary}[theorem]{Corollary}
\newtheorem{lemma}[theorem]{Lemma}
\newtheorem{example}[theorem]{Example}

\theoremstyle{remark}
\newtheorem{remark}[theorem]{Remark}

\def\<{\langle}
\def\>{\rangle}

\begin{document}
\title{\bf{The Fenchel-type inequality in the 3-dimensional Lorentz space and a Crofton formula}}

\author{Nan Ye%
  \thanks{School of Mathematical Sciences, Peking University, Beijing 100871, People's Republic of China. \texttt{yen@pku.edu.cn}. }
\and Xiang Ma%
  \thanks{LMAM, School of Mathematical Sciences, Peking University,
 Beijing 100871, People's Republic of China. \texttt{maxiang@math.pku.edu.cn}, Fax:+86-010-62751801. Corresponding author. Supported by NSFC 11471021}
  \and Donghao Wang%
  \thanks{School of Mathematical Sciences, Peking University, Beijing 100871, People's Republic of China.  \texttt{wdh\_nick\_123456@126.com}.}   }

\date{\today}
\maketitle

\begin{center}
{\bf Abstract}
\end{center}

We generalize the Fenchel theorem to strong spacelike (which means that the tangent vector and the curvature vector span a spacelike 2-plane at each point) closed curves with index 1 in the 3-dimensional Lorentz space, showing that the total curvatures must be less than or equal to $2\pi$. A similar generalization of the Fary-Milnor theorem is also obtained. We establish the Crofton formula on the de Sitter 2-sphere which implies the above results.

\hspace{2mm}


{\bf MSC(2000):\hspace{2mm} 52A40, 53C50, 53C65, }

\section{Introduction}

The total curvature of closed space curves (and submanifolds) is a classical topic in global differential geometry and topology. The Fenchel theorem \cite{Fenchel} says that in $\mathbb{R}^3$ there is always $\int k\mathrm{d}s\geq 2\pi$, and equality is attained exactly for convex plane curves.
The Fary-Milnor theorem \cite{Milnor} says that for nontrivial knot the lower bound is $4\pi$. (See also the appendix of \cite{Chern} for an enlightening discussion.)

The main purpose of this paper is to generalize these beautiful results to the 3-dimensional Lorentz space $\mathbb{R}^3_1$. For that task, it is crucial to find the appropriate assumptions and correct upper/lower bounds beforehand. We clarify these issues by making several observations as below.\\

\textbf{(1) The strong spacelike condition.}~~First, we would like the curve to have a moving frame at every point of the same type, so that not only there exist closed examples, but also one can define the length and the total curvature in a consistent way. A basic fact is that in $\mathbb{R}^3_1$, there exist neither closed timelike curve, nor closed spacelike curves with timelike normals (see Section~2 for definitions of these terms and Remark~\ref{rem-Ntimelike}). So it is natural to consider only the curves whose tangent vector $T$ and curvature vector $\kappa N$ span a spacelike 2-plane at each point.
We call this class as \emph{strong spacelike curves}, because they are not only spacelike curves, but also with spacelike osculating planes. In particular, we do not allow inflection points where the curvature vanishes. (So we can assume that the curvature satisfies $\kappa>0$.)\\

\textbf{(2) The index $I$: a well-defined global invariant.}~~ Different from the Euclidean 3-space, there is a well-defined notion of \emph{index} for closed spacelike curves in the Lorentz 3-space. This is the winding number of the tangent indicatrix (the image of the unit tangent vector $T$) around the de Sitter sphere
$\mathbb{S}^2_1=\{X\in \mathbb{R}^3_1:\langle X,X\rangle=1\}$ (the usual one-sheet hyperboloid, which is homotopy equivalent to a circle). It is integer-valued and always assumed to be positive. In the special case that $I=1$, the closed curve winds around some timelike axis exactly for one cycle; intuitively its shape is just a mild perturbation of a convex plane curve, whose projection to any spacelike plane is still convex (see Remark~\ref{rem-convex}).\\

\textbf{(3) The reversed inequality. }~~When we estimates the total curvature in this Lorentz setup and compare with plane curves, one should expect a reversed inequality in contrast with the Euclidean case. One reason is that the total curvature is equal to the length of the tangent indicatrix $T(S^1)$. A normal variation of $T(S^1)$ in $\mathbb{S}^2_1$ is always along a timelike direction, which decreases $L(T(S^1))$ for a spacelike geodesic. (A line of altitude generally has length greater than $2\pi$. But it is located on one side of the equator and is contained in a half space; hence it can not integrate to be a closed curve in $\mathbb{R}^3_1$.)\\

Now we can state the first of our main theorems as below.

\begin{theorem}[The Fenchel theorem in $\mathbb{R}^3_1$]
\label{thm-fenchel}
Let $\gamma$ be a closed strong spacelike curve in $\mathbb{R}^3_1$ with index $1$. Then its total curvature $\int_\gamma k\mathrm{d}s\leq2\pi$. The equality holds if and only if it's a convex curve on a spacelike plane.
\end{theorem}

Some people (as we did before) might naively expect that a closed strong spacelike curve in $\mathbb{R}^3_1$ with index $I\ge 2$ will have an upper bound $2I\pi$ for its total curvature. Yet this is not true: the total curvature in the general case for $I\ge 2$ can be arbitrarily large (like the example of a clam-shell curve in Section~4).

What if the closed curve in $\mathbb{R}^3_1$ becomes more complicated, not only with higher index $I>1$, but also being a nontrivial knot?
When $I=2$ we have the following interesting generalization of the Fary-Milnor theorem :

\begin{theorem}[The generalized Fary-Milnor theorem in $\mathbb{R}^3_1$]
\label{thm-farymilnor}
Let $\gamma$ be a closed strong spacelike curve in $\mathbb{R}^3_1$ with index $2$. Suppose it is a nontrivial knot. Then its total curvature $\int_\gamma k\mathrm{d}s<4\pi$.
\end{theorem}

One of the main technical difficulties in proving these theorems is the loss of compactness. For example, a powerful tool in proving both Fenchel theorem and Fary-Milnor theorem in $\mathbb{R}^3$ is the Crofton formula in the integral geometry, which gives an estimation of the length of an arc $\Gamma$ on $S^2$:
\begin{equation*}
\int_{S^2}n(Y^\perp)dA=4L(\Gamma).
\end{equation*}
Here $Y^\perp$ is the oriented great circle in $S^2$ with $Y$ as its pole, $n(Y^\perp)$ is the number of intersections of $Y^\perp$ with $\Gamma$, and $L(\Gamma)$ is the length of $\Gamma$. The key point is that $S^2$, the image space of the tangent indicatrix and the moduli space of the oriented great circles, is compact, endowed with the $\mathrm{SO}(3)$-invariant metric. In contrast, for a strong spacelike curve in the Lorentz 3-space, the tangent indicatrix $T(\gamma)$ is in the non-compact de Sitter sphere $\mathbb{S}^2_1$; the moduli space of the oriented great circles in $\mathbb{S}^2_1$ is the hyperbolic plan $\mathbb{H}^2$ which is also non-compact. So it seems no hope to establish a similar theory of integral geometry at the first glance.

Fortunately, when the pole $Y$ is very close to the lightcone (equivalently, when $Y$ tends to $\partial\mathbb{H}^2$, the boundary at infinity of the hyperbolic plane), the intersection number $n(Y^\perp)$ is a constant $2I$ where $I$ is the index of $\Gamma\subset \mathbb{S}^2_1$ (Lemma~\ref{lem-2I}). Thus we can still establish the Crofton formula in the de Sitter sphere as below (Theorem~\ref{thm-global}):
\begin{equation}
L(\Gamma)-2I\pi=-\frac{1}{2}\int_{\mathbb{H}^2}
(n(Y^\perp)-2I)\mathrm{d}Y.
\end{equation}
Here the integrand at the right hand side always has compact support.
Theorem~\ref{thm-fenchel} and \ref{thm-farymilnor} follows easily from this.

This paper is organized as follows. In Section~2 we derive the Crofton formula for piecewise smooth spacelike curves on $\mathbb{S}^2_1$. Applying this formula to the tangent indicatrix $T(\gamma)$ of a closed strong spacelike curve $\gamma$ in the Lorentz 3-space, we will obtain the reversed Fenchel inequality and the generalized Fary-Milnor theorem in Section~3. The discussions of the upper/lower bound of total curvature and illustrating examples are left to Section~4.

\section{The Crofton formula on the de Sitter sphere}

At the beginning we briefly review some standard definitions and fix the notations in a Lorentz space.
Let $\mathbb{R}^3_1$ be the 3-dimensional real linear space endowed with the Lorentz metric $\langle \cdot,\cdot\rangle$. In a fixed coordinate system, we may write
\[\langle X,Y\rangle=x_1y_1+x_2y_2-x_3y_3, ~~~X=(x_1,x_2,x_3), Y=(y_1,y_2,y_3).\]
A vector $X$ is called spacelike (lightlike, timelike) if $\langle X,X\rangle>0 (=0,<0)$, respectively. A subspace $V$ is called spacelike (lightlike, timelike) if the Lorentz inner product restricts to be a positive definite (degenerate, Lorentz) product on $V$. A curve or a surface in $\mathbb{R}^3_1$ is said to be spacelike (lightlike, timelike) if its tangent space at each point is so.

There are two different generalizations of the unit sphere in the Lorentz space: \emph{the de Sitter sphere} (the hyperboloid of one sheet)
\begin{equation*}
\mathbb{S}^2_1=\{X\in\mathbb{R}^3_1|\langle X,X\rangle=1\},
\end{equation*}
and \emph{the hyperbolic plane} (the upper half of the hyperboloid of two sheet)
\begin{equation}
\mathbb{H}^2=\{(x_1,x_2,x_3)\in\mathbb{R}^3_1|x_1^2+x_2^2-x_3^2=-1, x_3>0\}.
\end{equation}
The induced metric on $\mathbb{S}^2_1$ is a Lorentz metric of constant curvature $+1$, and $\mathbb{H}^2$ has an induced Riemannian metric with constant curvature $-1$. Besides them we have \emph{the light cone} consisting of lightlike vectors. The Lorentz group $O(2,1)$ acts on all of them as isometries.

Similar to the great circles on the round sphere $S^2\subset \mathbb{R}^3$, any spacelike geodesic on $\mathbb{S}^2_1$ is the intersection of $\mathbb{S}^2_1$ with a spacelike subspace of $\mathbb{R}^3_1$; it is always closed.

\begin{definition}
The unit normal vector (timelike, future directed) of the spacelike plane is called \emph{the pole} of the corresponding closed geodesic, which take values in $\mathbb{H}^2$. If we denote the pole by $Y$, the oriented closed spacelike geodesic will be denoted as $Y^\perp$. We will also abuse the notation $Y^\perp$ to denote the 2-dimensional subspace orthogonal to $Y$ if this will not cause confusion.
\end{definition}

Thus $Y,-Y$ will correspond to the same closed geodesic with opposite orientations. This establishes a 1-1 correspondence between the hyperbolic plane $\mathbb{H}^2$ and the space of oriented closed spacelike geodesics on $\mathbb{S}^2_1$.

$\mathbb{S}^2_1$ has a standard parametrization in terms of latitude and longitude $(\varphi,\theta)$ as below:
\[
(x_1,x_2,x_3)=
(\cosh\varphi\cos\theta,\cosh\varphi\sin\theta,\sinh\varphi).
\]
$\mathbb{H}^2$ has a similar expression $(x_1,x_2,x_3)=
(\sinh r\cos\theta,\sinh r\sin\theta,\cosh r).$
To overcome the difficulty of the non-compactness, let us consider a hyperbolic closed geodesic disk with radius $R$:
\begin{equation}
\mathbb{H}^2_R\triangleq\{(x_1,x_2,x_3)\in\mathbb{H}^2|1\leq x_3\leq\cosh R\}.
\end{equation}
By restricting to a compact domain we can establish the following local version of the Crofton formula. Notice that the hyperbolic geodesic disk $\mathbb{H}^2_R$, as well as the latitude and longitude parameters, are all defined under a fixed coordinate system, which will be used throughout the statement and the proof of the following theorem.

\begin{theorem}\label{thm-local}
[\textbf{Localized Crofton formula in a given coordinate system}].\\
Let $\Gamma$ be a piecewise smooth spacelike curve on $\mathbb{S}^2_1$ whose longitude ranges $\Delta\theta$ (with respect to the given coordinate). Then for a large enough radius $R$, we have:
\begin{equation}\label{eq-local}
\int_{\mathbb{H}^2_R}n(Y^\perp)\mathrm{d}Y=2\cosh R\cdot \Delta\theta-2L(\Gamma).
\end{equation}
Here $Y^\perp$ is the closed geodesic with $Y$ as its pole, $n(Y^\perp)$ is the number of intersections of $Y^\perp$ with $\Gamma$ (counted with multiplicity, yet without sign), and $L(\Gamma)$ is the length of $\Gamma$.
\end{theorem}
To prove this theorem, let us find out the subset of the poles $Y\in\mathbb{H}^2$ so that $Y^\perp$ has non-empty intersection with $\Gamma$.
Without loss of generality, write $\Gamma=\{e_1(s)\in \mathbb{S}^2_1|s\in[0,L]\}$ as a smooth curve with arc-length parameter $s$ and length $L=L(\Gamma)$, where
\begin{equation}\label{eq-e1}
e_1(s)=(\cosh\varphi(s)\cos\theta(s),\cosh\varphi(s)\sin\theta(s),
\sinh\varphi(s)).
\end{equation}
Set
\begin{eqnarray}
e_2(s)&=&(-\sin\theta(s),\cos\theta(s),0),\\
e_3(s)&=&(\sinh\varphi(s)\cos\theta(s),
\sinh\varphi(s)\cos\theta(s),\cosh\varphi(s)).\label{eq-e3}
\end{eqnarray}
Then $\{e_1(s),e_2(s),e_3(s)\}$ forms an orthonormal frame in $\mathbb{R}^3_1$ along $\Gamma$. The structure equations read
\begin{eqnarray}\label{eq-de1}
\begin{cases}
\frac{\mathrm{d}e_1}{\mathrm{d}s}=\cosh\varphi(s)\theta'(s)e_2
+\varphi'(s)e_3,\\
\frac{\mathrm{d}e_2}{\mathrm{d}s}=-\cosh\varphi(s)\theta'(s)e_1
+\sinh\varphi(s)\theta'(s)e_3,\\
\frac{\mathrm{d}e_3}{\mathrm{d}s}=\varphi'(s)e_1+\sinh\varphi(s)
\theta'(s)e_2.
\end{cases}
\end{eqnarray}
Since $\Gamma$ is spacelike and $s$ is the arc-length parameter, we have:
\begin{equation}\label{eq-theta}
\cosh^2\varphi(s)\cdot\theta'(s)^2-\varphi'(s)^2=1.
\end{equation}

\begin{remark}\label{rem-convex}
As a corollary to \eqref{eq-theta}, there is always $\theta'(s)\ne 0$.
Without loss of generality, suppose $\theta'(s)>0$ everywhere. This immediately implies that the projection of $\Gamma$ to the $Ox_1x_2$-plane is a plane curve whose tangent vector rotates in a monotonic way, hence a convex plane curve in a general sense. In particular, for a closed strong spacelike curve of index $I=1$, its orthogonal projection to a spacelike plane is always a strictly convex Jordan curve.
\end{remark}

By \eqref{eq-theta}, there exists a function $\tau(s)$ such that
\begin{equation}\label{eq-tau}
\cosh\varphi(s)\theta'(s)=\cosh\tau(s),~\varphi'(s)=\sinh\tau(s).
\end{equation}
For a closed spacelike geodesic $Y^\perp$ intersecting $\Gamma$ at $e_1(s)$, its pole can be written as
\begin{equation}\label{eq-Y}
Y=\sinh\psi e_2(s)+\cosh\psi e_3(s).
\end{equation}
Note that $(s,\psi)$ gives a natural parametrization of the image of such poles.
Restricted to $\mathbb{H}^2_R$, by \eqref{eq-e3} there must be
\begin{equation}\label{eq-Ybound}
Y\in\mathbb{H}^2_R~~\Leftrightarrow~~
\cosh\psi\cosh\varphi(s)\leq\cosh R.
\end{equation}
Let's differentiate $Y$ to find its geometric information:
\begin{eqnarray}\label{eq-dY2}
\begin{cases}
\frac{\partial Y}{\partial\psi}=\cosh\psi e_2+\sinh\psi e_3,\\
\frac{\partial Y}{\partial s}=\sinh(\tau(s)-\psi)e_1+\sinh\varphi(s)\theta'(s)\frac{\partial Y}{\partial\psi}.
\end{cases}
\end{eqnarray}
The pull-back of the area element from the hyperbolic plane $\mathbb{H}^2$ by $Y$ is
\begin{equation}\label{eq-dY1}
\mathrm{d}Y=\sinh|\tau(s)-\psi|\mathrm{d}\psi \mathrm{d}s.
\end{equation}
\begin{remark}\label{rem-degenerate}
We can read out the information of the pull-back metric by $Y$ from \eqref{eq-dY2} and \eqref{eq-dY1}. It degenerates when $\tau(s)=\psi$.
By \eqref{eq-Y}, \eqref{eq-tau} and the first formula in \eqref{eq-de1}, we see that this happens exactly when the closed geodesic $Y^\perp$ intersects with $\Gamma$ non-transversely.
\end{remark}

\begin{proof}[\textbf{Proof to Theorem~\ref{thm-local}}]~~
Rewrite the left hand side of \eqref{eq-local} as an iterated integral:
\begin{equation}\label{eq-integral}
\int_{\mathbb{H}^2_R}n(Y^\perp)\mathrm{d}Y=\int_0^L \mathrm{d}s\int_{\psi_1(s)}^{\psi_2(s)}
\sinh|\tau(s)-\psi|\mathrm{d}\psi.
\end{equation}
We need to choose the radius $R$ large enough by the following considerations:
\begin{enumerate}
\item To determine the lower and upper limit $\psi_1(s),\psi_2(s)$ for any given $s$, note that the poles
$\{Y|Y^\perp\cap\Gamma=e_1(s)\}$ form a subset of $\mathbb{H}^2$, which we hope to has nonempty intersection with $\mathbb{H}^2_R$. By \eqref{eq-Ybound} one needs only to choose $R$ so large that
\begin{equation}\label{eq-Rbound1}
\text{max}_{s\in[0,L]}\cosh\varphi(s)<\cosh R.
\end{equation}
Then for any $s$, the inequality \eqref{eq-Ybound} always has solutions $\psi\in [\psi_1(s),\psi_2(s)]$ with \begin{equation}\label{eq-coshpsi}
\cosh\psi_i(s)=\frac{\cosh R}{\cosh\varphi(s)}.
\end{equation}
\item To determine the sign of $\tau(s)-\psi_i$, note that
$\cosh\varphi(s)$ is bounded on $s\in [0,L]$, hence $\cosh\psi_i(s)$ increases with $R$ according to
\eqref{eq-coshpsi}, yet $\cosh\tau(s)$ is uniformly bounded by
\eqref{eq-tau}. Choose $R$ so large that \begin{equation}\label{eq-Rbound2}
\text{max}_{s\in[0,L]}\cosh^2\varphi(s)\theta'(s)<\cosh R.
\end{equation}
Then $\psi_1(s)<\tau(s)<\psi_2(s)$ for each $s$.
\end{enumerate}
Now the first integral in \eqref{eq-integral} can be evaluated as below:
\begin{eqnarray*}
\int_{\psi_1(s)}^{\psi_2(s)}\sinh|\tau(s)-\psi|\mathrm{d}\psi
&=&\int_{\psi_1(s)}^{\tau(s)}\sinh(\tau(s)-\psi)\mathrm{d}\psi
+\int_{\tau(s)}^{\psi_2(s)}\sinh(\psi-\tau(s))\mathrm{d}\psi\\
&=&\cosh(\tau(s)-\psi_1(s))+\cosh(\psi_2(s)-\tau(s))-2\\
&=&\cosh\tau(s)(\cosh\psi_1(s)+\cosh\psi_2(s))-2\\
&=&2\cosh R\cdot\theta'(s)-2.
\end{eqnarray*}
Here we used \eqref{eq-coshpsi} and \eqref{eq-tau} in the last step. The conclusion follows easily.
\end{proof}

\begin{corollary}
Let $\Gamma$ be a closed piecewise smooth spacelike curve on $\mathbb{S}^2_1$ with index $I$ (the longitude of $\Gamma$ ranges $2I\pi$). For $R$ large enough there is
\begin{equation}\label{eq-closed}
\int_{\mathbb{H}^2_R}n(Y^\perp)\mathrm{d}Y=4I\pi\cosh R-2L(\Gamma).
\end{equation}
\end{corollary}

The results above are not quite satisfying for us, because they are only meaningful with respect to a fixed coordinate system, and some artificial constant $R$ is involved. So we proceed to derive a geometrically invariant version.

\begin{theorem}\label{thm-global}
[\textbf{Global Crofton formula in $\mathbb{S}^2_1$}] ~~Notations as above. Consider a closed piecewise smooth spacelike curve on $\mathbb{S}^2_1$ with index $I\ge 1$.
\begin{enumerate}
\item For radius $R$ large enough,
\begin{equation}\label{eq-2IR}
\int_{\mathbb{H}^2_R}(n(Y^\perp)-2I)\mathrm{d}Y=-2L(\Gamma)+4I\pi.
\end{equation}
\item When $R\rightarrow+\infty$, the integral is still meaningful and we have
\begin{equation}\label{eq-2I}
L(\Gamma)-2I\pi=-\frac{1}{2}\int_{\mathbb{H}^2}(n(Y^\perp)-2I)
\mathrm{d}Y.
\end{equation}
\end{enumerate}
\end{theorem}
\begin{proof}
The area of the hyperbolic disc is given by
\begin{equation}\label{eq-area}
\int_{\mathbb{H}^2_R}\mathrm{d}Y=2\pi(\cosh R-1).
\end{equation}
Combining \eqref{eq-area} and \eqref{eq-closed} we obtain the first conclusion. The second one follows if we can show $n(Y^\perp)-2I$ is zero outside a compact domain. This is guaranteed by the next lemma.
\end{proof}

\begin{lemma}\label{lem-2I}
Let $\Gamma$ be a given closed smooth spacelike curve on $\mathbb{S}^2_1$ with index $I>0$ and $Y^\perp$ is a 2-dimensional subspace orthogonal to $Y$. Then the intersection number $n(Y^\perp)=\sharp\{\Gamma\cap Y^\perp\}$ must be $2I$ in the following two cases:
\begin{enumerate}
\item $Y^\perp$ is timelike or lightlike (i.e., $Y$ is spacelike or lightlike);
\item $Y^\perp$ is spacelike and close enough to the lightcone (i.e., $Y$ is timelike, and the corresponding point on $\mathbb{H}^2$ is contained in a suitable neighborhood of $\partial\mathbb{H}^2$.)
\end{enumerate}
\end{lemma}

\begin{proof}
Without loss of generality, suppose that $\Gamma$ is parameterized by \eqref{eq-e1} and $Y=(0,1,a)$ with $a\geq0$. Then the intersection points correspond exactly to the zeros of $\sin\theta(s)-a\tanh\varphi(s)$.

Recall that there is a 1-1 correspondence between the longitude parameter $\theta\in[-\frac{\pi}{2},2I\pi-\frac{\pi}{2})$ and the arc-length parameter $s\in[0,L)$ when $\Gamma$ has index $I$. We will show that on each interval $\theta\in(-\frac{\pi}{2}+n\pi,\frac{\pi}{2}+n\pi)$ for $n=0,1,\cdots,2I-1$ there is a unique zero of $\sin\theta(s)-a\tanh\varphi(s)$. We prove this when $n=0$ (for other values of $n$ the proof is similar). Equivalently we consider the zeros of the continuous function
\begin{equation}\label{eq-zero}
f_a(s)\triangleq \theta(s)-\arcsin(a\tanh\varphi(s)).
\end{equation}
Note that $|a\tanh\varphi(s)|< 1$ when $0\le a\le 1$ in the first case (for timelike or lightlike 2-plane), or when $1<a<\text{min}_{s\in[0,L]}\{\frac{1}{\tanh\varphi(s)}\}$ in the second case.

It is clear that $f_a<0$ when $\theta=-\frac{\pi}{2}$, and $f_a>0$ when $\theta=\frac{\pi}{2}$.
So $f_a$ has at least one zero on $(-\frac{\pi}{2},\frac{\pi}{2})$. We will show $f_a$ is monotonically increasing on this interval by computing its derivative:
\[
f_a'(s)
=\theta'(s)-\frac{\varphi'(s)/\cosh\varphi}
{\sqrt{\frac{\cosh^2\varphi}{a^2}-\sinh^2\varphi}}
\ge \theta'(s)-\frac{|\varphi'(s)|}{\cosh\varphi}>0
\]
when $a\le 1$. In the last step we used \eqref{eq-theta}.
If $a>1$, we can still choose $a$ so close to $1$ such that \[\frac{1}{\sqrt{\frac{\cosh^2\varphi}{a^2}-\sinh^2\varphi}}
<1+\epsilon
<\theta'(s)\frac{\cosh\varphi(s)}{|\varphi'(s)|}
=\frac{1}{|\tanh\tau(s)|}
\]
holds uniformly on $[0,L]$ for some small $\epsilon>0$. Thus in all cases we have proved that $f_a$ is a monotonic function on $(-\frac{\pi}{2},\frac{\pi}{2})$ with a unique zero.
The rest is easy and the proof is finished.
\end{proof}

Lemma~\ref{lem-2I} is a natural fact from a geometric viewpoint. When the subspace $Y^\perp$ is timelike or lightlike, a closed spacelike curve $\Gamma$ must wind around the $x_3$ axis in a monotonic way and intersect $Y^\perp$ exactly two times in one cycle, hence with $2I$ intersection points when the index is $I$. If we perturb $Y$ to be a timelike vector and $Y^\perp$ to be spacelike, this intersection number $n(Y^\perp)=\sharp\{\Gamma\cap Y^\perp\}$ should be the same constant $2I$ if the perturbation is small enough, because $\Gamma$ is compact and $n(Y^\perp)$ is an integer depending on $Y$ continuously. This observation led us to the discovery of this Crofton formula.

\begin{remark}
We can provide another conceptual proof to Lemma~\ref{lem-2I}. Consider the following set-valued mapping $\Phi_\Gamma$ from $\mathbb{R}^3_1\setminus\{0\}$ to subsets of $\Gamma$:
\[
\Phi_\Gamma(Y)\triangleq\Gamma\cap Y^\perp=
\{\Gamma(s)|\langle e_1(s),Y\rangle=0\}.
\]
Notice that the critical points of this mapping $\Phi$ is nothing but the locus of the co-normal vectors of the closed curve $\Gamma$ as given by
\[\{Y|\langle e_1(s),Y\rangle=0=\langle e_1'(s),Y\rangle,
~~\text{for some}~s\in [0,L]\}.
\]
These timelike lines encloses a cone $\Omega$ inside the lightcone. Its complement in $\mathbb{R}^3_1\setminus\{0\}$ is a connected domain $\Omega^*$ which also includes all the spacelike and (nonzero) lightlike vectors. $\Phi_\Gamma$ is a regular set-valued mapping from $\Omega^*$ to finite subsets of $\Gamma$. In particular, the counting function $\sharp(\Phi_\Gamma(Y))$ (counting the intersection number) is a continuous function with positive integer values. This must be a constant, which is easy to show to be $2I$.
\end{remark}

\begin{remark}
After finishing this work, we noticed that Solanes and Teufel have obtained more general results on similar Crofton formulas for all Lorentz space forms \cite{Solanes}.
\end{remark}

\section{Application to closed spacelike curves in $\mathbb{R}^3_1$}

In this section we will apply the Crofton formula \eqref{eq-2I} to the tangent indicatrix of a closed strong spacelike curve in $\mathbb{R}^3_1$.

\begin{definition}
Let $\gamma:(a,b)\rightarrow R^{3}_{1}$ be a spacelike curve. We call $\gamma$ a strong spacelike curve if its osculating plane is spacelike and the curvature $\kappa$ is positive everywhere. (This includes closed convex plane curves on a spacelike plane as a special case.)
\end{definition}

The Frenet frame $\{T,N,B\}$ for a strong spacelike curve is defined similar to the Euclidean case with the following Frenet formula:
\begin{equation}
\frac{d}{d s}\left(
                               \begin{array}{c}
                                 T \\
                                 N \\
                                 B \\
                               \end{array}
                             \right)
                             =\left(
                                \begin{array}{ccc}
                                  0  & k    & 0    \\
                                  -k & 0    & \tau \\
                                  0  & \tau & 0    \\
                                \end{array}
                              \right)
                              \left(
                                \begin{array}{c}
                                  T \\
                                  N \\
                                  B \\
                                \end{array}
                              \right)~.
\end{equation}
Here $N,B$ are the normal and the bi-normal vectors, respectively, and $\tau$ is the torsion. Note that the coefficient matrix is no longer skey-symmetric because the inner product here is not positive definite.

\begin{remark}\label{rem-Ntimelike}
It is easy to see that there exist no regular closed spacelike curves in $R^{3}_{1}$ with timelike normal vector everywhere.
Otherwise, the tangent indicatrix will be a timelike curve on the de Sitter sphere. But on $\mathbb{S}^2_1$ there are no such closed timelike curves.
\end{remark}

Only for a strong spacelike curve $\gamma$ (regular and closed)
the total curvature has a natural definition, which is also the length of its tangent indicatrix $T(\gamma)\subset \mathbb{S}^2_1$. By assumption, we can apply the global Crofton formula \eqref{eq-2I} to $T(\gamma)$, which is a spacelike closed curve with index $I=1$.

\begin{proof}[\textbf{Proof of the Fenchel theorem \ref{thm-fenchel}}]

Because $T=\gamma'(s)$ and $\int_{\gamma}(T,Y)\mathrm{d}s=0$ for any fixed (timelike) vector $Y$, this forces that $T(\gamma)$ can not be contained in one side of any geodesic. So it must intersect every closed geodesic with at least two intersection points. Now $n(Y^\perp)\geq 2$ in \eqref{eq-2I} for each $Y\in\mathbb{H}^2$, which implies
\[
\int_\gamma k\mathrm{d}s=L(T(\gamma))\leq2\pi.\qedhere
\]
\end{proof}

\begin{proof}[\textbf{Proof of the generalized Fary-Milnor theorem \ref{thm-farymilnor}}]

By \eqref{eq-2I} we need only to show that the tangent indicatrix has at least four intersections with each closed geodesic, i.e. $n(Y^\perp)\geq4$ for each $Y\in\mathbb{H}^2$.

Otherwise, suppose there exists a specific $Y\in\mathbb{H}^2$ such that the tangent indicatrix $T(\gamma)$ has exactly two intersections with $Y^\perp$. In other words, the height function $\langle \gamma,Y\rangle$ has exactly two critical points, one maximum and the other being the minimum, which separate the curve $\gamma$ into two arcs. Connecting the points on each arc with the same value of $\langle \gamma,Y\rangle$ by a line segment. This provides a disc with $\gamma$ as the boundary. This contradicts to the assumption that $\gamma$ is a nontrivial knot.
\end{proof}

\section{General results on the range of the total curvature}

After knowing the generalizations of those classical results in the last section, the reader might be curious about the value distribution of the total curvature for general closed strong spacelike curves.
In this section, we will show that there is neither positive lower bounds nor upper bounds in the general case.\\

It is easy to see that the infimum of the total curvature is $0$ for closed strong spacelike curves in $\mathbb{R}^3_1$. For concrete examples we can find such a $\gamma$ whose tangent indicatrix is symmetric with respect to the $Ox_1x_3$-plane and $Ox_2x_3$-plane which is a perturbation of a space quadrilateral formed by four lightlike line segments on $\mathbb{S}^2_1$.\\

For strong spacelike curve with index $I\geq2$, a naive guess is that the total curvature has an upper bound $2I\pi$. It turns out surprisingly that the total curvature can be arbitrary large if there is no further restriction. We realized this possibility only after noticing the following phenomena:
\begin{enumerate}
\item A line of latitude has total length $\geq2\pi$, and this length can be arbitrarily large;
\item When $I\ge 2$ for a closed spacelike curve on $\mathbb{S}^2_1$, it is not difficult to make it intersecting with every geodesic.
\end{enumerate}
A simple example with arbitrarily large total curvature and being strong spacelike, closed, with index $2$ is given below.

\begin{example}\label{ex-shell-1}
Let $\gamma(\theta)=(\cos\theta,\sin\theta,h(\theta)),
\theta\in[0,4\pi]$, where
\begin{eqnarray}
h(\theta)=
\begin{cases}
\epsilon\theta,  &\theta\in[0,\frac{\pi}{2}]\\
\frac{\pi}{2}\epsilon-\epsilon\cos\theta, &
\theta\in[\frac{\pi}{2},\frac{3}{2}\pi]\\
-\epsilon(\theta-2\pi), &\theta\in[\frac{3}{2}\pi,\frac{5}{2}\pi]
\\
-\frac{\pi}{2}\epsilon+\epsilon\cos\theta,&
\theta\in[\frac{5}{2}\pi,\frac{7}{2}\pi]\\
\epsilon(\theta-4\pi),&\theta\in[\frac{7}{2}\pi,4\pi]
\end{cases}
\end{eqnarray}
and $\epsilon\in(0,1)$.

\begin{figure}[!htb]
\centering\begin{tabular}{cc}
\begin{minipage}[t]{2in}
\includegraphics[width=2in]{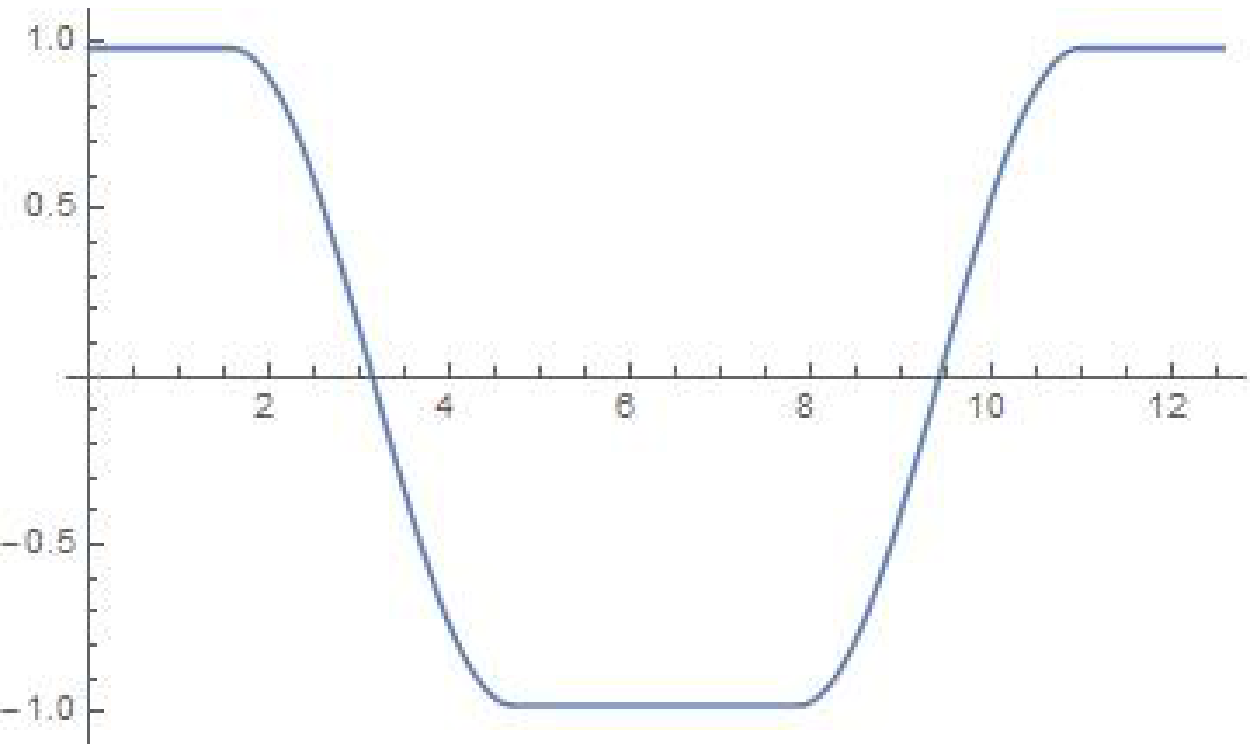}
\caption{Graph of $h'(\theta)$}
\end{minipage} 
\begin{minipage}[t]{2in}
\includegraphics[width=2in]{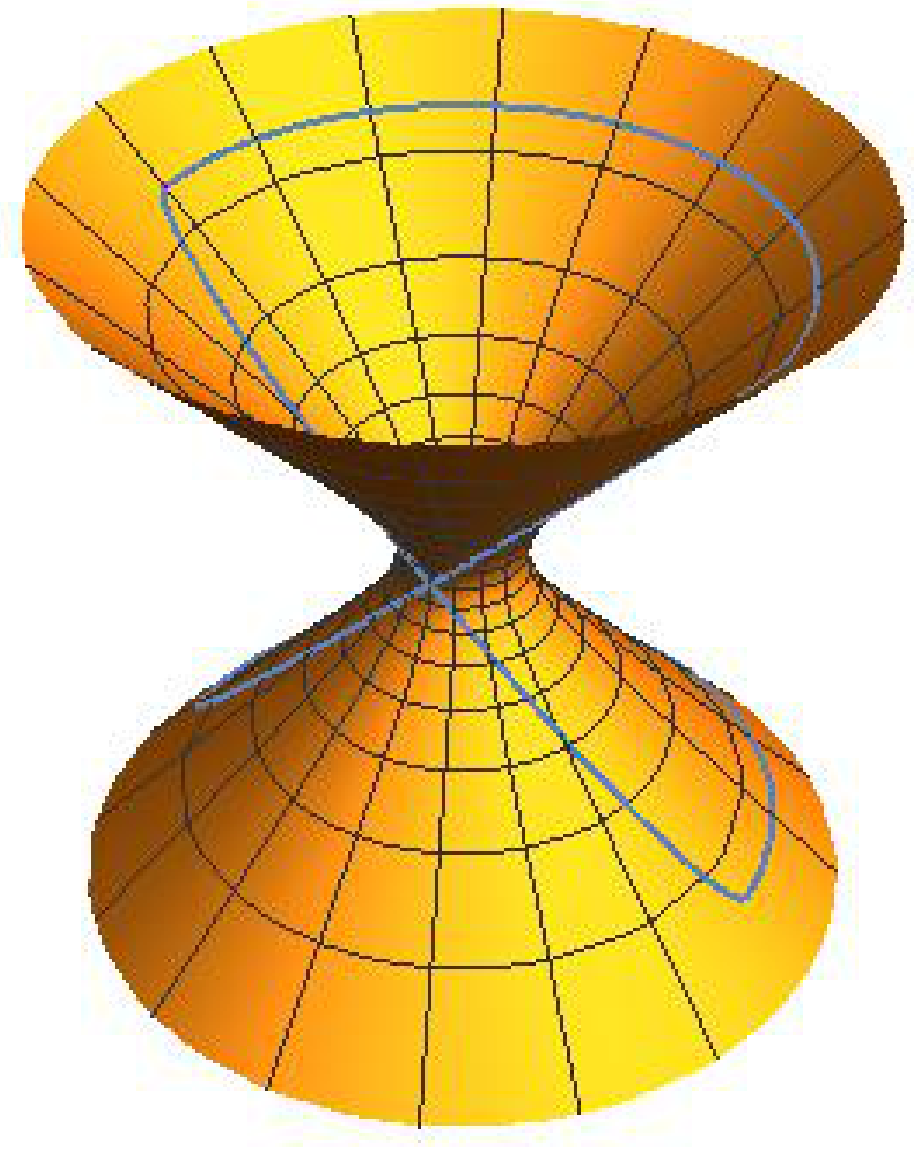}
\caption{The indicatrix $T(\gamma)$}
\end{minipage}
\end{tabular}
\end{figure}
Then
\begin{eqnarray}
h'(\theta)=
\begin{cases}
\pm\epsilon\sin\theta,&
\theta\in[\frac{\pi}{2},\frac{3}{2}\pi]\cup
[\frac{5}{2}\pi,\frac{7}{2}\pi],\\
\pm\epsilon,& \text{otherwise}.
\end{cases}
\end{eqnarray}
and
\begin{eqnarray}
h''(\theta)=
\begin{cases}
\pm\epsilon\cos\theta,&
\theta\in[\frac{\pi}{2},\frac{3}{2}\pi]\cup
[\frac{5}{2}\pi,\frac{7}{2}\pi],\\
0,& \text{otherwise}.
\end{cases}
\end{eqnarray}
So $h'(\theta)^2+h''(\theta)^2=\epsilon^2<1$ and $\gamma(\theta)$ is strong spacelike. It is closed with $I=2$ and total curvature
\[
\int_\gamma k\mathrm{d}s\ge 4\int_0^{\frac{\pi}{2}}
\frac{\sqrt{1-(h')^2-(h'')^2}}{1-(h')^2}d\theta =\frac{2\pi}{\sqrt{1-\epsilon^2}}.
\]
When $\epsilon\rightarrow1^-$, it is self-evident that $\int_\gamma k\mathrm{d}s$ has no upper bounds.
\end{example}

As a final remark, let us compare two concrete examples of closed curves of index $I=2$: a trefoil knot and a clam-shell curve as illustrated below.
\begin{figure}[!htb]
\centering\begin{tabular}{cc}
\begin{minipage}[t]{2in}
\includegraphics[width=1.5in]{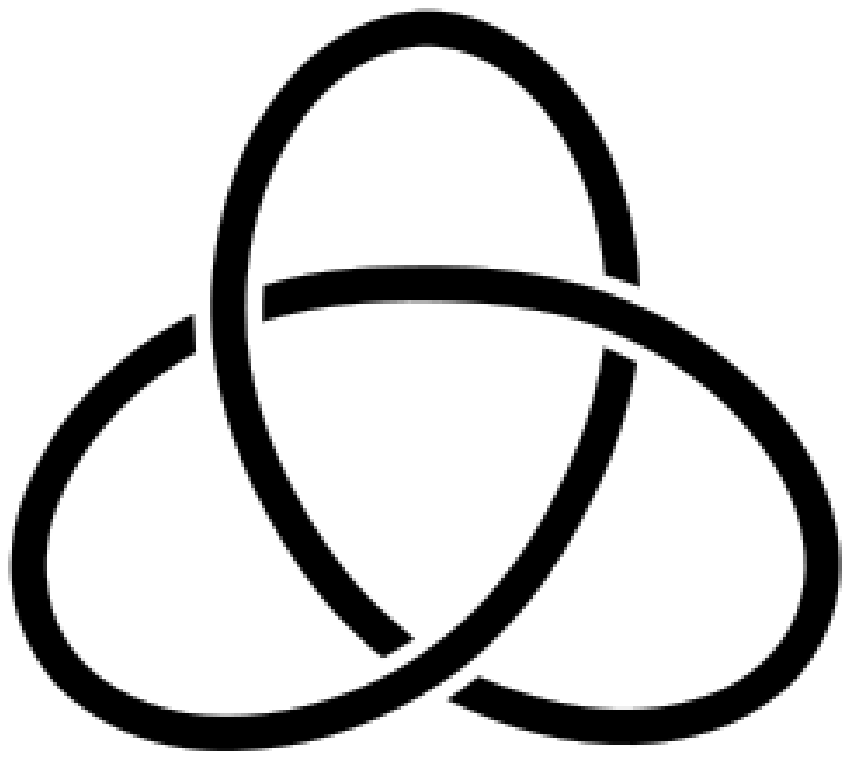}
\caption{Trefoil knot}
\end{minipage} 
\begin{minipage}[t]{2in}
\includegraphics[width=1.5in]{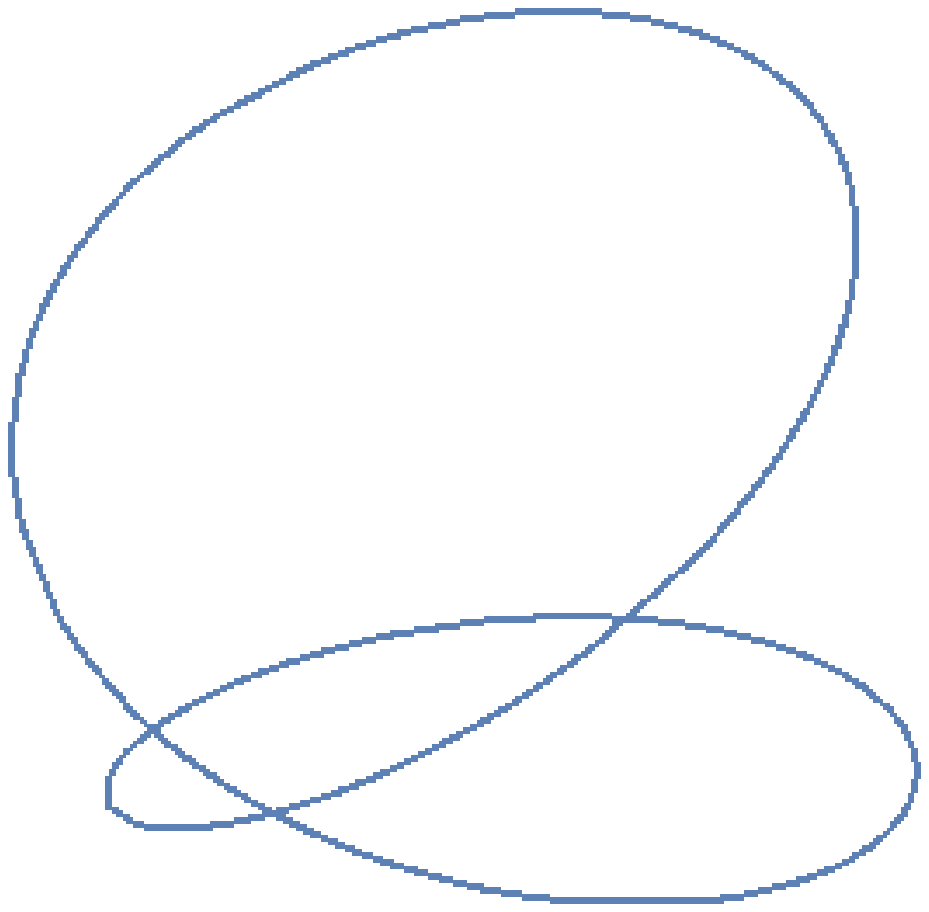}
\caption{clam-shell curve}
\end{minipage}
\end{tabular}
\end{figure}
Notice that a trefoil knot can be realized as a strong spacelike closed curve with $I=2$ in $\mathbb{R}^3_1$ (just imagine the trefoil knot diagram on the spacelike plane $Ox_1x_2$ and make a slight perturbation along the $x_3$ direction).
According to the generalized Fary-Milnor theorem \ref{thm-farymilnor}, it must have total curvature less than $4\pi$. Our explanation is that as a nontrivial knot, the tangent indicatrix of a nontrivial knot is restricted to be close to some spacelike plane which can not be perturbed quite freely.

On the other hand, a clam-shell curve in $\mathbb{R}^3_1$ like Example~\ref{ex-shell-1} can have large total curvature. It can be divided into two symmetric parts, both approximating a helix in their major parts; the tangent indicatrix is very close to a half latitude circle in the corresponding part whose length can be very long.

\end{document}